\documentclass[a4paper,twoside,american]{amsart}
\usepackage[T1]{fontenc}
\usepackage[latin9]{inputenc}
\usepackage{color}
\usepackage{babel}
\usepackage{verbatim}
\usepackage{textcomp}
\usepackage{amsbsy}
\usepackage{amstext}
\usepackage{amsthm}
\usepackage{amssymb}
\usepackage{esint}
\usepackage{fancyhdr}
\usepackage{geometry}
\usepackage[unicode=true,
 bookmarks=true,bookmarksnumbered=false,bookmarksopen=true,bookmarksopenlevel=1,
 breaklinks=false,pdfborder={0 0 0},pdfborderstyle={},backref=false,colorlinks=true]
 {hyperref}
\hypersetup{pdftitle={Some measure rigidity and equidistribution results for $\beta$-maps},
 pdfauthor={Nevo Fishbein},
 linkcolor=blue, citecolor=magenta}

\makeatletter

\pdfpageheight\paperheight
\pdfpagewidth\paperwidth

\numberwithin{equation}{section}
\theoremstyle{plain}
\newtheorem{thm}{\protect\theoremname}[section]
\theoremstyle{definition}
\newtheorem{problem}[thm]{\protect\problemname}
\theoremstyle{plain}
\newtheorem{cor}[thm]{\protect\corollaryname}
\theoremstyle{plain}
\newtheorem{prop}[thm]{\protect\propositionname}
\theoremstyle{plain}
\newtheorem{lem}[thm]{\protect\lemmaname}

\makeatother

\providecommand{\corollaryname}{Corollary}
\providecommand{\lemmaname}{Lemma}
\providecommand{\problemname}{Problem}
\providecommand{\propositionname}{Proposition}
\providecommand{\theoremname}{Theorem}

\makeatletter
\let\ps@plain\ps@fancy   
\makeatother

\begin{document}
\global\long\def\r{\mathbb{R}}%
\global\long\def\z{\mathbb{Z}}%
\global\long\def\q{\mathbb{Q}}%
\global\long\def\n{\mathbb{N}}%
\global\long\def\rmz{\mathbb{R}/\mathbb{Z}}%
\global\long\def\ev#1{\textbf{E}\left(#1\right)}%
\global\long\def\ee#1#2{\textbf{E}_{#1}\left(#2\right)}%
\global\long\def\rf#1{\left\lceil #1\right\rceil }%
\global\long\def\fr#1{\left[#1\right]}%
\global\long\def\sy{\mathcal{B}}%
\global\long\def\no#1{\left\Vert #1\right\Vert }%
\global\long\def\con{\mathcal{C}}%
\global\long\def\pr{\textbf{P}}%
\global\long\def\var#1{\textbf{Var}\left(#1\right)}%

\renewcommand*{\thefootnote}{\fnsymbol{footnote}}

\newsavebox{\overlongequation} \newenvironment{dontbotheriftheequationisoverlong}  {\begin{displaymath}\begin{lrbox}{\overlongequation}$\displaystyle}  {$\end{lrbox}\makebox[0pt]{\usebox{\overlongequation}}\end{displaymath}}

\newcommand{\diam}{\mathop{}\!\mathrm{diam}}

\newcommand{\essup}{\mathop{}\!\mathrm{ess\,sup}}

\newcommand{\mjo}{\mathop{}\!\mathrm{mjo}}
\title[]{Some measure rigidity and equidistribution results for $\beta$-maps}
\author{Nevo Fishbein \\ \baselineskip=1.5\baselineskip \MakeLowercase{(e-mail: {nevofi@gmail.com}) }}

\maketitle

\begin{abstract}
We prove $\times a$ $\times b$ measure rigidity for multiplicatively
independent pairs when $a\in\n$ and $b>1$ is a ``specified'' real
number (the $b$-expansion of $1$ has a tail or bounded runs of $0$'s)
under a positive entropy condition. This is done by proving a mean
decay of the Fourier series of the point masses average along $\times b$
orbits. We also prove a quantitative version of this decay under stronger
conditions on the $\times a$ invariant measure. The quantitative
version together with the $\times b$ invariance of the limit measure
is a step toward a general Host-type pointwise equidistribution theorem
in which the equidistribution is for Parry measure instead of Lebesgue.
We show that finite memory length measures on the $a$-shift meet
the mentioned conditions for mean convergence. Our main proof relies
on techniques of Hochman.
\end{abstract}

\section{Introduction}

\subsection{Definitions and notation}

First, we introduce a couple of frequently used definitions. Given
a positive real number $b$ we define the $b$-fold map $T_{b}:\left[0,1\right)\rightarrow\left[0,1\right)$
by $T_{b}\left(x\right)=b\cdot x\mod1$. We identify $\left[0,1\right)$
with $\rmz$ so that $T_{b}$ is referred interchangeably as a toral
map that has at most one discontinuity at $0$. For a real pair $\left(s,t\right)\in\left(1,\infty\right)\times\left(1,\infty\right)$
we say that they are \emph{multiplicatively independent} and write
$s\nsim t$ if $\frac{\log s}{\log t}\notin\q$. As is customary,
we write $\fr{\cdot},\rf{\cdot}$ for the floor and the\textbf{ }ceiling
functions respectively and $\left\{ x\right\} =x-\fr x$ for the fractional
part of a non-negative real number $x$.

For a Polish space $X$ denote by $\left(X,\Sigma\right)$ the Borel
space that is associated to it. Let $\mu$ be some probability measure
on $\left(X,\Sigma\right)$. A sequence of points $\left\{ x_{n}\right\} _{n=1}^{\infty}$
in $\left(X,E\right)$ is said to be \emph{equidistribute for $\mu$}
if the mean of their point masses weakly-{*} converges to $\mu$:
$\frac{1}{N}\sum_{n=1}^{N}\delta_{x_{n}}\xrightarrow{w-*}\mu$. Let
$T:X\rightarrow X$ be a function. If for some $x\in X$ we have $\frac{1}{N}\sum_{n=1}^{N}\delta_{T^{n}x}\xrightarrow{w-*}\mu$
then we say that $x$ \emph{equidistributes for $\mu$ under $T$}.

For every $m\in\z$ and $x\in\rmz$ we write $e_{m}\left(x\right)=e\left(mx\right)=\exp\left(2\pi imx\right)$.
Let $\mu$ be a finite Borel measure on $\rmz$ and $M$$\left(\rmz\right)$
be the set of all such measures. The \emph{$m$th Fourier coefficient
of $\mu$} is defined to be $\hat{\mu}\left(m\right)=\intop e_{m}\left(x\right)d\mu\left(x\right)$,
the \emph{Fourier transform of $\mu$} is the sequence $\hat{\mu}=\left(\hat{\mu}\left(n\right)\right)_{n\in\z}$
and the map $\mathcal{F}:M\left(\rmz\right)\rightarrow\ell^{\infty}\left(\z\right)$
defined by $\mathcal{F}\left(\mu\right)=\hat{\mu}$ is the \emph{Fourier
transform}.

We denote the Lebesgue measure on $\rmz$ by $m$ and the integration
with respect to Lebesgue measure by $dz$. In our context, all the
absolutely continuous measures are with respect to Lebesgue. 

\subsection{Background}

Furstenberg's Diophantine theorems \cite{key-2} formed the background
to Furstenberg's pioneering $\times2$ $\times3$ conjecture about
the measure rigidity of $T_{2},T_{3}$. This conjecture suggests that
the only non-atomic and ergodic Borel probability measure on the circle
which is also $T_{2}$ and $T_{3}$ invariant is Lebesgue measure. 

The best result so far on this problem was proved by Rudolph and later
strengthened further by Johnson, now known as Rudolph-Johnson theorem
\cite{key-4,key-5}. It establishes the conjecture for every multiplicatively
independent pair of integers $m,n\geq2$ under the additional assumption
of positive entropy. 

Later on the following pointwise 'equidistributional' version was
proved by Host \cite{key-1} when $\gcd\left(a,b\right)=1$ and improved
to the case $a\nsim b$ by Hochman and Shmerkin \cite{key-7}: Let
$\mu$ be a probability measure on $\rmz$ which is invariant, ergodic
and has positive entropy w.r.t. an endomorphism $T_{a}$. Then $\mu$-a.e.
$x$ equidistributes for Lebesgue measure under every endomorphism
$T_{b}$ with $a\nsim b$. We write HET (Host's equidistribution theorem)
for short when referring to this theorem. 

In another direction the next result is due to Parry \cite{key-6}:
Given a real $b>1$ there exists a unique $T_{b}$-invariant Borel
probability measure that is equivalent to Lebesgue measure. Its Radon-Nikodym
derivative can be written explicitly as,
\[
f\left(x\right)={\displaystyle \sum_{x<T_{b}^{n}\left(1\right)}}\frac{1}{b^{n}}\text{ and }1-\frac{1}{b}\leq f\leq\frac{1}{1-\frac{1}{b}}
\]

Therefore for an integer $a\geq2$ and a non-integer $b>1$ there
doesn't exist a joint $T_{a}$ and $T_{b}$ invariant and absolutely
continuous probability measure.

This suggests the following problem:
\begin{problem}
\label{prob:(measure_rigidity)} Let $a\geq2$ be an integer and let
$b>1$ be a non integer such that $a\nsim b$. Is it true that there
are no non-atomic and ergodic Borel probability measures on the circle
which are both $T_{a}$ and $T_{b}$ invariant?
\end{problem}

The HET and the Parry measure may also be related to each other via
the next possible generalization of HET:
\begin{problem}
\label{prob:(Generalized-HET)}(Generalized HET) Let $\mu$ be a probability
measure on $\rmz$ which is invariant, ergodic and has positive entropy
w.r.t. an endomorphism $T_{a}$. Let $1<b\in\r$ with $a\nsim b$.
Is it true that $\mu$-a.e. $x$ equidistributes for Parry measure
under $T_{b}$?
\end{problem}

These two problems are the main concern of this paper.

\subsection{\label{subsec:Results}Results}

We begin by outlining our strategy for tackling problems \ref{prob:(measure_rigidity)}
and \ref{prob:(Generalized-HET)} from a harmonic analysis perspective.
We denote the space of bi-infinite sequences whose both limits are
zero by $c_{0}\left(\z\right)$.

In the case of Problem \ref{prob:(measure_rigidity)} a possible strategy
is the following:
\begin{enumerate}
\item To show that for every probability measure $\mu$ on $\rmz$ which
is invariant, ergodic and has positive entropy w.r.t. an endomorphism
$T_{a}$, it holds that $\mu$-a.e. $x$,
\[
\left\{ \intop{\displaystyle \limsup_{N\rightarrow\infty}}\left|\frac{1}{N}\sum_{i=0}^{N-1}e_{m}\left(T_{b}^{i}\left(x\right)\right)\right|d\mu\right\} _{m=-\infty}^{\infty}\in c_{0}\left(\z\right)
\]
\item Assume by contradiction that $\mu$ is a non-atomic and ergodic Borel
probability measure on the circle which is $T_{a}$ and $T_{b}$-invariant
as in Problem \ref{prob:(measure_rigidity)}. Notice that,
\begin{align*}
\left|\hat{\mu}\left(m\right)\right| & \leq\intop{\displaystyle \limsup_{N\rightarrow\infty}}\left|\frac{1}{N}\sum_{i=0}^{N-1}e_{m}\left(T_{b}^{i}\left(x\right)\right)\right|d\mu
\end{align*}
and therefore by the first step,
\[
\left\{ \left|\hat{\mu}\left(m\right)\right|\right\} _{m=-\infty}^{\infty}\in c_{0}\left(\z\right)
\]
But the $T_{a}$ invariance implies that for every $m\in\z$ we have
$\hat{\mu}\left(m\right)=\hat{\mu}\left(a\cdot m\right)$ and the
limit above possible only when, 
\[
\hat{\mu}\left(m\right)=\begin{cases}
1 & \text{for }m=0\\
0 & \text{for }m\neq0
\end{cases}
\]
i.e. when $\mu$ is Lebesgue measure. This contradicts the assumption
that $\mu$ is $T_{b}$-invariant because Lebesgue measure is not
$T_{b}$-invariant.
\end{enumerate}
For the case of Problem \ref{prob:(Generalized-HET)} we will need
another definition. Let $a,b$ and $\mu$ as in Problem \ref{prob:(Generalized-HET)}.
Given some $x\in\left[0,1\right)$ and some subsequence $\left\{ N_{k}\right\} $
of $\n$ we denote $\lambda_{x,\left\{ N_{k}\right\} }=\lim_{k\rightarrow\infty}\frac{1}{N_{k}}\sum_{i=0}^{N_{k}-1}\delta_{T_{b}^{i}\left(x\right)}$,
if it exists. It is not hard to show that the only $T_{b}$-invariant
measure with $\ell^{2}\left(\z\right)$ Fourier transform is the Parry
measure. Therefore, a possible strategy to solve Problem \ref{prob:(Generalized-HET)}
is the following:
\begin{enumerate}
\item To show that $\mu$-a.e. $x$,
\[
\left\{ {\displaystyle \limsup_{N\rightarrow\infty}}\left|\frac{1}{N}\sum_{i=0}^{N-1}e_{m}\left(T_{b}^{i}\left(x\right)\right)\right|\right\} _{m=-\infty}^{\infty}\in\ell^{2}\left(\z\right)
\]
\item To show that for every $\mu$-typical $x$, $\lambda_{x,\left\{ N_{k}\right\} }$
is necessarily $T_{b}$-invariant (for every $\left\{ N_{k}\right\} $
such that the limit exists).
\item By the uniqueness of Parry measure we conclude that it is the limit
of every proper convergent subsequence of the sequence 
\[
\left\{ \frac{1}{N}\sum_{i=0}^{N-1}\delta_{T_{b}^{i}\left(x\right)}\right\} _{N=1}^{\infty}
\]
and we know that such a convergent subsequence does exist. Thus, the
whole sequence also converges to Parry measure.
\end{enumerate}
Before stating our first result we give another definition. Given
$1<b\in\r$ there is a non-surjective measurable embedding of $\left[0,1\right)$
in $\Lambda_{b}^{\n}=\left\{ 0,1,\dots,\fr b-1\right\} ^{\n}$ which
is given by the \emph{$b$-expansion}, 
\[
r\mapsto\left(\left[b\cdot r\right],\left[b\cdot T_{b}r\right],\left[b\cdot T_{b}^{2}r\right]\dots\right)
\]
That is, a subset of $\Lambda_{b}^{\n}$ can represent $\left[0,1\right)$.
We say that a real positive $b$ is a \emph{specified number} or with
the \emph{specification property} if the $b$-expansion of $1$ has
bounded runs of $0$'s or has a tail of $0$'s. The set of specified
numbers is uncountable and dense in $\left(1,\infty\right)$ (see
Section \ref{subsec:beta-shifts}). In Section \ref{sec:pf_of_Rigidity}
we prove the following:
\begin{thm}
\label{thm:Raj} Let $\mu$ be a probability measure on $\rmz$ which
is invariant, ergodic and has positive entropy w.r.t. an endomorphism
$T_{a}$ for some positive integer $a\geq2$ and let $b>1$ be a real
specified number such that $a\nsim b$. Then,
\[
\left\{ \intop{\displaystyle \limsup_{N\rightarrow\infty}}\left|\frac{1}{N}\sum_{i=0}^{N-1}e_{m}\left(T_{b}^{i}\left(x\right)\right)\right|d\mu\right\} _{m=-\infty}^{\infty}\in c_{0}\left(\z\right)
\]
\end{thm}

According to our strategy, Theorem \ref{thm:Raj} yield a Furstenberg-type
measure rigidity for the class of specified $b$'s,
\begin{cor}
\label{co:rigidity} Let $a\geq2$ be an integer and $b>1$ be a specified
non integer such that $a\nsim b$. Then there are no jointly $T_{a}$
and $T_{b}$ invariant non-atomic and ergodic Borel probability measures
with positive entropy under $T_{a}$.
\end{cor}

This is an answer to Problem \ref{prob:(measure_rigidity)} for the
class of specified $b$'s under an entropy assumption (the same assumption
currently needed when $b\in\n$).

To explain our next result we need a few more definitions. For convenience
we denote $\Lambda_{a}^{-\n}=\Omega^{-}$. For an integer $a\geq2$,
let $\mathcal{A}$ denote the $a$-adic partition of $\left[0,1\right)$:
$\left\{ \left[\frac{k}{a},\frac{k+1}{a}\right)\right\} _{k=0}^{a-1}$
and $\mathcal{\mathcal{A}}\left(x\right)\in\mathcal{\mathcal{A}}$
denote the element which contains $x\in X$. Let $\widetilde{\Omega}=\Omega^{-}\times\left[0,1\right)$
be the natural extension of $\left(\rmz,\mu,T_{a}\right)$ together
with the map $\widetilde{T}_{a}\left(\omega,x\right)=\left(\omega\mathcal{A}\left(x\right),T_{a}x\right)$
and write $\widetilde{\mu}$ for the unique extension of $\mu$ to
a $\widetilde{T}_{a}$-invariant measure on $\widetilde{\Omega}$.
Let $\mathcal{C}=\bigvee_{i=-\infty}^{0}{\displaystyle \widetilde{T}_{a}^{i}\mathcal{A}}$
denote the $\sigma$-algebra in $\widetilde{\Omega}$ generated by
projection to the past: $\widetilde{\omega}=\left(\omega,x\right)\mapsto\omega$
and let $\left\{ \widetilde{\mu}_{\omega}^{\mathcal{C}}\right\} _{\omega}$
be the corresponding disintegration. Notice that the members of $\left\{ \widetilde{\mu}_{\omega}^{\mathcal{C}}\right\} _{\omega}$
depend only on the $\Omega^{-}$-component and that we can identified
them as measures on $\left[0,1\right)$ such that $\mu=\intop\mu_{\omega}d\widetilde{\mu}\left(\omega\right)$.

In Section \ref{part:Pf_quant_Raj} under stronger assumptions we
prove a quantitative mean version of Theorem \ref{thm:Raj}:
\begin{thm}
\label{thm:quant_Rajchman} Let $a,b,\mu$ be as in Theorem \ref{thm:Raj}
and let $\left\{ I_{j}^{n}\right\} _{j=1}^{n}$ be the uniform partition
of the unit interval into $n$ sub-intervals. Assume that there exists
some positive $\alpha$, such that for every $n$ the following holds,
\[
\underset{\eta,j}{\essup}\mu_{\eta}\left(I_{j}\right)\leq O\left(n^{-\alpha}\right)
\]
and such that for some positive $\beta$, for every positive integer
$n$ we have, 
\[
\underset{\eta}{\essup}\iint\chi_{\left\{ x:\left|x-y\right|<n^{-1}\right\} }d\mu_{\eta}\left(x\right)d\mu_{\eta}\left(y\right)\leq O\left(n^{-\beta}\right)
\]

Then for every integer $m$,
\[
\ee{\mu}{\limsup_{N}\left|\frac{1}{N}\sum_{i=0}^{N-1}e_{m}\left(T_{b}^{i}x\right)\right|}<O\left(\left|m\right|^{\min_{\gamma,\delta>0}\max\left\{ -\delta\alpha,\frac{-1-\delta\left(\alpha-1\right)+\gamma\delta}{2},\frac{\delta\left(1-\gamma\beta\right)}{2}\right\} }\right)
\]
\end{thm}

We remark that one can always assume $\alpha\le\beta$ (see Section
\ref{part:Pf_quant_Raj}).

In Section \ref{subsec:Finite-memory-length} we show that the conditions
on $\mu$ in Theorem \ref{thm:quant_Rajchman} are flexible enough
to work for every process with memory of finite length such as variable
length mixing Markov chains. Nonetheless, it certainly does not always
work as we show in Section \ref{subsec:A-counterexample}.

Our secondary result from Section \ref{part:The-invariance-of} says
that,
\begin{thm}
\label{thm:xb-inv.} For every real $b>1$ and a $\mu$-typical $x$,
if $\left\{ 0\right\} $ is not an atom of a partial limit $\lambda_{x,\left\{ N_{k}\right\} }=\lim_{k\rightarrow\infty}\frac{1}{N_{k}}\sum_{i=0}^{N_{k}-1}\delta_{T_{b}^{i}}\left(x\right)$,
then $\lambda_{x,\left\{ N_{k}\right\} }$ must be $T_{b}$-invariant.
\end{thm}

In particular, Theorem \ref{thm:Raj} together with Wiener's lemma
\cite{key-23} imply that $\left\{ 0\right\} $ is not an atom of
$\lambda_{x,\left\{ N_{k}\right\} }$, so Theorem \ref{thm:xb-inv.}
proves step 2 in the second strategy under the assumptions of Theorem
\ref{thm:Raj}.

Notice that since 
\[
-0.5<\min_{\gamma,\delta>0}\max\left\{ -\delta\alpha,\frac{-1-\delta\left(\alpha-1\right)+\gamma\delta}{2},\frac{\delta\left(1-\gamma\beta\right)}{2}\right\} <0
\]
 Our strategy fails to solve Problem \ref{prob:(Generalized-HET)}.
Let us briefly examine some other alternatives. We focused on trying
to relax the $\ell^{2}\left(\z\right)$ condition as follows.

Let $L^{1}\left(m\right)$ be the set of absolutely continuous measures
and let $M_{c}\left(\rmz\right)$ be the set of continuous measures.
We say that a measure $\mu$ on $\rmz$ is \emph{Rajchman} if $\lim_{\left|n\right|\rightarrow\infty}{\displaystyle \hat{\mu}\left(n\right)=0}$
and denote the set of Rajchman measures by $\mathcal{R}$. The Riemman-Lebesgue
lemma implies that every absolutely continuous measure is Rajchman.
Works by several mathematicians at the beginning of the 20th century
(see \cite{key-17}) established the more extensive result that $L^{1}\left(m\right)\subsetneq\mathcal{R}\subsetneq M_{c}\left(\rmz\right)$.
We can naturally ask: Is it true that Rajchman $T_{b}$-invariant
measure must be absolutely continuous? Equivalently, is it true that
$L^{1}\bigl(m\bigr)\cap\bigl\{ T_{b}\text{-invariant}\bigr\}=\mathcal{R}\cap\bigl\{ T_{b}\text{-invariant}\bigr\}$?
Some evidence for a positive answer comes from the special case of
an integer $b$: In this case the Parry measure is just the Lebesgue
measure and if $\mu$ is $T_{b}$-invariant and Rajchman then as we
already observed for every $n\in\z$, $\hat{\mu}\left(T_{b}n\right)=\hat{\mu}\left(b\cdot n\right)$.
Thus, besides $\hat{\mu}\left(0\right)$ all the other Fourier coefficients
must vanish and $\mu$ is indeed Lebesgue. However, in Section \ref{part:Rajchman+invariant is insufficient}
we show that,
\begin{prop}
There exists a $T_{b}$-invariant Rajchman measure that is not absolutely
continuous.
\end{prop}

\subsection{Acknowledgments}

First, I would like to thank Prof. B. (Benjy) Weiss for many helpful
discussions about dynamics, and especially for his suggestion regarding
the example in Section \ref{subsec:A-counterexample}. An earlier
version of this study was submitted at the Hebrew University as part
of an academic degree. It was done under formal guidance of Prof.
M. Hochman, for which I am thankful. Lastly, I also would like to
thank Prof. Z. Wang for a thorough review of this work and Prof. S.
Baker for providing advice on the example in Section \ref{part:Rajchman+invariant is insufficient}. 

\section{Preliminaries}

\subsection{\label{subsec:beta-shifts}$\beta$-shifts}

Recall our notation $\Lambda_{b}=\left\{ 0,1,\dots,\fr b-1\right\} $.
With respect to the product $\sigma$-algebra on $\Lambda_{b}^{\n}$,
the\emph{ shift transformation} $\sigma:\Lambda_{b}^{\n}\rightarrow\Lambda_{b}^{\n}$
which is defined by $\sigma\left(\left(\lambda_{i}\right)\right)=\left(\lambda_{i+1}\right)$
turns $\Lambda_{b}^{\n}$ into a dynamical system that we call the\emph{
full $\fr b$ shift}. The restriction of the full shift to the closure
of the subset of sequences which encodes $b$-expansions is a subshift
that we call the $b$-shift and denote by $X_{b}\subset\Lambda_{b}^{\n}$.

One says that a real positive $b$ is,
\begin{itemize}
\item A \emph{simple number} if the $b$-expansion of $1$ has a $0$'s
tail.
\item A \emph{simple Parry number} if the $b$-expansion of $1$ is a periodic
sequence. In some sources it is also called a (purely) \emph{periodic
number}.
\item A \emph{Parry number} if the $b$-expansion of $1$ has a periodic
tail. In some sources it is also called an\textbf{ }\emph{eventually
periodic number}.
\end{itemize}
It's immediate to conclude that,
\[
\left\{ \text{simple \#s}\right\} \subset\left\{ \text{simple Parry \#s}\right\} \subset\left\{ \text{Parry \#s}\right\} \subset\left\{ \text{specified \#s}\right\} 
\]
Parry showed \cite{key-6} that the simple numbers are everywhere
dense in $\left(1,\infty\right)$, and so is the set of specified
numbers. Schmeling showed \cite{key-10} that the set of specified
numbers also has Hausdorff dimension $1$, but it is meager and has
Lebesgue measure $0$. In particular, it has the cardinality of the
continuum. 

An important property of specified numbers is,
\begin{prop}
\label{prop:specification_prop}When $b$ has the specification property,
the orbit of $1$ under $T_{b}$ (in $\left[0,1\right)$) remains
bounded away from $0$ unless it hits it.
\end{prop}

\begin{proof}
Let $1<b\in\r$ be a specified number. We write $b_{0}=\fr b,\ b_{1}=\fr{b\left\{ b\right\} },\dots$
and similarly we write $r_{0}=T_{b}^{1}\left(1\right)=\left\{ b\right\} ,\ r_{1}=T_{b}^{2}\left(1\right)=\left\{ b\left\{ b\right\} \right\} ,\dots$.
We need to prove that $0<\inf_{n}\left\{ r_{n}:r_{n}>0\right\} $.
The special case of simple $b$ is trivial. If we assume by contradiction
that $0=\inf_{n}\left\{ r_{n}:r_{n}>0\right\} $ then $b$ can't be
simple and there's an upper bound $k\in\n$ on the length of runs
of $0$'s. But for every $k$ there exists a $n_{0}\in\n$ with $r_{n_{0}}<b^{-\left(k+1\right)}$
and therefore $b_{n_{0}+i}=0$ for $1\leq i\leq k+1$ in contradiction.
\end{proof}

Finally, we present a result by Parry \cite{key-6} that gives a criterion
to determine whether a given sequence $\left(b_{n}\right)\in\left\{ 0,\dots\left[b\right]\right\} ^{\mathbb{N}}$
is a $b$-expansion of some $x\in\left[0,1\right)$, $x=b_{0}+\frac{b_{1}}{\beta}+\dots$.
We emphasize that there might be many representaions of $x$ in this
form but only one of them corresponds to the $b$-expansion that we
described earlier. This will be useful for constructing the counterexample
in Section \ref{part:Rajchman+invariant is insufficient}.

If $\left(a_{0},a_{1},\dots\right),\left(b_{0},b_{1},\dots\right)$
are sequences of the same length (finite or infinite) of nonnegative
integers less than $b$, we write $\left(a_{0},a_{1},\dots\right)<\left(b_{0},b_{1},\dots\right)$
when $a_{n}<b_{n}$ for the first $a_{n}\neq b_{n}$.
\begin{thm}
\label{thm:Parry's criterion}(Parry's criterion) Let $b>1$ be a
non simple number. If the $b$-expansion of $b$ is $b=a_{0}+\frac{a_{1}}{b}+\dots$
and $\left(b_{0},b_{1},\dots\right)$ is a sequence of nonnegative
integers, a necessary and sufficient condition for the existence of
$x$ with $b$-expansion, $x=b_{0}+\frac{b_{1}}{b}+\dots$ is that
$\left(b_{n},b_{n+1},\dots\right)<\left(a_{0},a_{1},\dots\right)$
for all $n\geq1$. In particular, $\left(a_{n},a_{n+1},\dots\right)<\left(a_{0},a_{1},\dots\right)$
for all $n\geq1$.
\end{thm}

\subsection{\label{sec:Specifics-in-ergodic and entropy}Entropy theory}

Let $\left(X,\sy\right)$ be a standard Borel space and $\mathcal{D}\subset\sy$
be a measurable partition. We write $\mathcal{\mathcal{D}}\left(x\right)\in\mathcal{\mathcal{D}}$
for the element which contains $x\in X$. This is also well defined
when $\mathcal{\mathcal{D}}$ is countably generated $\sigma$-algebra.
In addition, we denote the join of two finite partitions $\mathcal{A},\mathcal{B}$
by $\mathcal{A\lor\mathcal{B}}=\left\{ A\cap B:A\in\mathcal{A},B\in\mathcal{B}\right\} $.
Let $\left(X,\mathcal{B},\mu\right)$ be a probability space then
the \emph{Shannon entropy} of $\mu$ w.r.t. a partition $\mathcal{A}$
of $X$ is the non-negative number $H_{\mu}\left(A\right)=-\sum_{A\in\mathcal{A}}\mu\left(A\right)\log\mu\left(A\right)$.
The \emph{entropy $h_{\mu}\left(T,\mathcal{A}\right)$ of a partition
$\mathcal{A}$} of a measure preserving system $\left(X,\mathcal{F},\mu,T\right)$
is the limit $h_{\mu}\left(T,\mathcal{A}\right)=\lim_{n\rightarrow\infty}\frac{1}{n}H_{\mu}\left(\mathcal{A}_{n}\right)$
where $\mathcal{A}_{n}={\displaystyle \bigvee{}_{i=0}^{n-1}}T^{-i}\mathcal{A}$.
The \emph{Kolmagorov-Sinai entropy} (or just the \emph{entropy}) of
the m.p.t. $\left(X,\mathcal{F},\mu,T\right)$ is $h_{\mu}\left(T\right)=\sup_{\mathcal{A}}h_{\mu}\left(T,\mathcal{A}\right)$
where the supremum is taken over all the finite partitions $\mathcal{A}$.
Equality is achieved if $\mathcal{A}$ is a generating partition,
$\mathcal{F}={\displaystyle \bigvee_{n=1}^{\infty}}\mathcal{A}_{n}\mod\mu$.

A landmark result in ergodic theory is the Shannon-McMillan-Breiman
theorem \cite{key-21}: Let $\left(X,\mathcal{F},\mu,T\right)$ be
an ergodic measure preserving system and $\mathcal{A}$ a finite partition.
Then $\mu$-a.e. $x$ $\lim_{n\rightarrow\infty}\frac{1}{n}\log\mu\left(\mathcal{A}_{n}\left(x\right)\right)=h_{\mu}\left(T,\mathcal{A}\right)$.
It is not hard to deduce from it that an ergodic and $T_{a}$-invariant
Borel probability measure $\mu$ with positive entropy is non-atomic
(this is also can be proved directly). We will use this corollary
in occasional places. 

\subsection{\label{subsec:General-results-on}General results on equidistribution
(due to Hochman)}

This subsection covers three results that we adopt from Hochman \cite{key-8}.
Two of them are presented with a very superficial description of their
proofs to help the reader gain some intuition. A more thorough treatment
can be found in the original paper.

We denote the real line translation and scaling maps by,
\begin{align*}
R_{\theta}x & =x+\theta\\
S_{t}x & =t\cdot x
\end{align*}
respectively. $R_{\theta}$ is taken $\mod1$ when acting on $\left[0,1\right)\cong\rmz$

Let $\mu$ be a probability measure on $\rmz$ and $E\in\sy$ such
that $\mu\left(E\right)>0$. We write $\mu_{E}=\frac{1}{\mu\left(E\right)}\cdot\mu|_{E}$
for the normalized restriction of $\mu$ to $E$.

The next technique relates orbits to the local structure of $\mu$,
\begin{thm}
\label{thm:2.2 of Hochman} \cite{key-8} Let $T:X\rightarrow X$
be a continuous map of compact metric space. Let $\mathcal{D}_{1},\mathcal{D}_{2},\dots$
be a refining sequence of finite Borel partitions. Let $\mu$ be a
Borel probability measure on $X$ and assume that $\sup_{n\in\n}\left\{ \diam T_{a}^{n}D:D\in\mathcal{D}_{n+k},\mu\left(D\right)>0\right\} \rightarrow0$
as $k\rightarrow\infty$. Then for $\mu$-a.e. $x$,
\[
\left(\frac{1}{N}\sum_{n=1}^{N}\delta_{T^{n}x}-\frac{1}{N}\sum_{n=1}^{N}T^{n}\mu_{\mathcal{D}_{n}\left(x\right)}\right)\xrightarrow[w-*]{N\rightarrow\infty}0
\]
\end{thm}

The idea of the proof is to take a countable dense set in $C\left(X\right)$
and prove the weak-{*} convergence with respect to its members. The
left average can be replaced with the $\mathcal{D}_{n+k}$-conditional
mean by the assumption. The right average is just the $\mathcal{D}_{n}\left(x\right)$-conditional
mean. A variant of the ergodic theorem for martingale differences
 implies that their limits are equal.

The second theorem is about equidistribution along orbits of the form
$(n\theta,T^{\left[\beta n\right]}x)$ where $x$ is a typical point
for $\mu$.
\begin{thm}
\label{thm:Corollary-2.7} \cite{key-8} Let $\left(X,\mu,T\right)$
be an ergodic m.p.s. on a compact metric space. Let $\beta>0$ and
$\theta\neq0$. Then for $\mu$\textup{-a.e.} $x$ the sequence $\left(n\theta,T^{\left[\beta n\right]}x\right)$
equidistributes for a measure $\nu_{x}$ on $\left[0,1\right)\times X$
that satisfies $\intop\nu_{x}d\mu\left(x\right)=\tau\times\mu$, where
$\tau$ is the invariant measure on $\left(\left[0,1\right),R_{\theta}\right)$
supported on the orbit closure of $0$.
\end{thm}

The proof uses an intermediate result which says that for $\tau\times\mu$-a.e.
$\left(u,x\right)$, the orbit $\left(n\theta+u,T^{\left[\beta n\right]}x\right)$
equidistributes for a measure $\nu_{u,x}$ on $\left[0,1\right)\times X$
satisfying,
\[
\intop\nu_{u,x}dz\times d\mu\left(u,x\right)=\tau\times\mu
\]
We won't go into its detail besides mentioning that a suspension of
height $1$ is used to overcome the integer part issue. Back to the
proof of Theorem \ref{thm:Corollary-2.7}, it implies that Lebesgue-a.e.
$u\in\left[0,1\right)$ and $\mu$-a.e. $x$, $\delta_{n\theta+u}\times\delta_{T^{\left[n\beta\right]}x}\xrightarrow{w-*}\nu_{u,x}$
and also that $\intop\nu_{u,x}dz\times\mu\left(u,x\right)=m\times\mu$.
By translation of the first coordinate we get that $\left(n\theta,T^{\left[\beta n\right]}x\right)$
equidistributes for a measure $\nu_{x}$ on $\left[0,1\right)\times X$
where $\nu_{x}=\nu_{x,0}$, and with Lemma 2.3 from the original paper
we find that $\intop\nu_{x}d\mu\left(x\right)=\tau\times\mu$.

Lastly, we give a slightly modified version for Hochman's evaluation
of the Fourier transform of scaled measures,
\begin{lem}
\label{lem:3.2} Let $\mu$ be a non-atomic probability measure on
$\r$. Then for every $\left(c,d\right)\subset\r$ and for every $r>0$
and $m\neq0$,
\[
\intop_{0}^{1}\left|\mathcal{F}\left(\left(S_{b^{z}}\mu\right)|_{\left[c,d\right]}\right)\left(m\right)\right|^{2}dz\leq\frac{2\mu\left(\left[c,d\right]\right)^{2}}{r\cdot\left|m\right|}+\intop_{c}^{d}\intop_{c}^{d}\chi_{B_{r}\left(y'\right)}\left(y\right)d\mu\left(y\right)d\mu\left(y'\right)
\]
where $B_{r}\left(x\right)=\left\{ y:\left|x-y\right|<r\right\} $. 
\end{lem}

\begin{proof}
(based on Hochman's proof \cite{key-8}) Using Fubini,
\begin{align*}
\intop_{0}^{1}\left|\mathcal{F}\left(\left(S_{b^{z}}\mu\right)|_{\left[c,d\right]}\right)\left(m\right)\right|^{2}dz & =\intop_{0}^{1}\left|\intop_{c}^{d}e\left(mb^{z}y\right)d\mu\left(y\right)\right|^{2}dz\\
 & =\intop_{0}^{1}\intop_{c}^{d}\intop_{c}^{d}e\left(mb^{z}y\right)\overline{e\left(mb^{z}y'\right)}d\mu\left(y\right)d\mu\left(y'\right)dz\\
 & =\intop_{0}^{1}\intop_{c}^{d}\intop_{c}^{d}e\left(mb^{z}\left(y-y'\right)\right)d\mu\left(y\right)d\mu\left(y'\right)dz\\
 & =\intop_{c}^{d}\intop_{c}^{d}\intop_{0}^{1}e\left(mb^{z}\left(y-y'\right)\right)dzd\mu\left(y\right)d\mu\left(y'\right)
\end{align*}
and then changing of variables $t=b^{z}$,
\begin{gather*}
\leq\intop_{c}^{d}\left(\intop_{\left[c,d\right]\backslash B_{r}\left(y'\right)}\left|\intop_{1}^{b}\frac{1}{\log\left(b\right)t}e\left(m\left(y-y'\right)t\right)dt\right|d\mu\left(y\right)+\intop_{B_{r}\left(y'\right)\cap\left[c,d\right]}1d\mu\left(y\right)\right)d\mu\left(y'\right)
\end{gather*}
Finally, using integration by parts for the inner integral in the
left summand,
\begin{gather*}
\leq\frac{2\mu\left(\left[c,d\right]\right)^{2}}{r\cdot\left|m\right|}+\intop_{c}^{d}\intop_{c}^{d}\chi_{B_{r}\left(y'\right)}\left(y\right)d\mu\left(y\right)d\mu\left(y'\right)
\end{gather*}
\end{proof}

\section{\label{sec:pf_of_Rigidity} Proof of theorem \ref{thm:Raj}}

Let $a,b,\mu$ be as in Theorem \ref{thm:Raj} and denote $\alpha=\frac{\log b}{\log a}$.
Given a positive integer $n$ denote $n'=\fr{\alpha n}$ and $z_{n}=\left\{ \alpha n\right\} =\alpha n\mod1$.
That is, $\left\{ z_{n}\right\} _{n\in\n}$ is the orbit of $0\in\rmz$
under the irrational rotation by $\alpha$. Recall that $\left(\widetilde{\Omega},\widetilde{\mu},\widetilde{T}_{a}\right)$
is the natural extension of $\left(\rmz,\mu,T_{a}\right)$ and $\mu=\intop\mu_{\omega}d\widetilde{\mu}\left(\omega\right)$
is the disintegration of $\mu$ given the past.

Recall that $\mathcal{A}$ denote the $a$-adic partition of $\left[0,1\right)$:
$\left\{ \left[\frac{k}{a},\frac{k+1}{a}\right)\right\} _{k=0}^{a-1}$
and correspondingly $\mathcal{A}_{n}=\bigvee_{i=0}^{n-1}T_{a}^{-i}\mathcal{A}$
is the $a$-adic partition of generation-$n$: $\left\{ \left[\frac{k}{a^{n}},\frac{k+1}{a^{n}}\right)\right\} _{k=0}^{a^{n}-1}$.
This simple partition is convenient to work with and it can easily
be shown that it is a generator for $T_{a}$. Naturally, $\mathcal{A}_{n}\left(x\right)$
stands for the $n$th-generation atom which contains $x$.

Let $f$ be a nonnegative piecewise linear function on $\left[0,1\right).$
We denote the set of its discontinuities. The \emph{minimal jumps
oscillation of $f$} is defined by, 
\[
\mjo\left(f\right)=\min_{x\in J}\left\{ \lim_{x'\in x^{+}}f\left(x'\right)-\lim_{x'\in x^{-}}f\left(x'\right)\right\} 
\]
 Recall the notation $r_{0}=T_{b}^{1}\left(1\right)=\left\{ b\right\} ,\ r_{1}=T_{b}^{2}\left(1\right)=\left\{ b\left\{ b\right\} \right\} ,\dots$.
Denote $m_{b}={\displaystyle \inf_{n}}\left\{ r_{n}:r_{n}>0\right\} $.
Since $b$ is specified, we have $0<m_{b}\leq1$ as shown in Proposition
\ref{prop:specification_prop}. Specifically, $\inf_{n}\mjo\left(\left\{ T_{b}^{n}\right\} \right)=m_{b}$.

Notice that for every $0<\theta<1-a^{-n'}$ the function $T_{b}^{n}\circ R_{\theta}\circ S_{a^{-n'}}$
is $T_{b}^{n}$ composed on the affine map $R_{\theta}\circ S_{a^{-n'}}\left(x\right)=a^{-n'}x+\theta$
of the real line. That is, this composition is $T_{b}^{n}$ stretched
horizontally by $a^{n'}$ and translated by $\theta$. Hence it is
well defined piecewise linear map with minimal jumps oscillation which
is greater or equal to $\mjo\left(T_{b}^{n}\right)$. In our notations
it means that $\mjo\left(T_{b}^{n}\circ R_{\theta}\circ S_{a^{-n'}}\right)\geq m_{b}$.
Notice that, $0<a^{-n'}\cdot b^{n}$<$a$ so $T_{b}^{n}\circ R_{\theta}\circ S_{a^{-n'}}$
also has a uniform slope bounded from above by $0<a$. These last
two properties imply that $T_{b}^{n}\circ R_{\theta}\circ S_{a^{-n'}}$
has at most $\left\lceil a/m_{b}\right\rceil $ discontinuities with
a minimal gap of $m_{b}/a$ between them. Thus, for a sufficiently
refined uniform partition of the unit interval, each member of the
partition contains at most one discontinuity.

Now we turn to prove Theorem \ref{thm:Raj}. Fix a $\widetilde{\mu}$-typical
$\omega\in\Omega^{-}$ and a $\mu_{\omega}$-typical $x\in\left[0,1\right)$.
Thus, we want to show asymptotic decay of the $\widetilde{\mu}$-expectation
of,
\begin{equation}
\limsup_{N}\left|\frac{1}{N}\sum_{n=1}^{N}e_{m}\left(T_{b}^{n}x\right)\right|\label{eq:first-1}
\end{equation}

It holds that $T_{b}^{n}\mathcal{A}_{n'+l}\left(x\right)$ has diameter
$O\left(a^{-l}\right)$ under the metric on $\left[0,1\right)\cong\rmz$
and by Theorem \ref{thm:2.2 of Hochman} we have,
\[
=\limsup_{N}\left|\frac{1}{N}\sum_{n=1}^{N}\intop e_{m}d\left(T_{b}^{n}\left(\mu_{\omega}\right)_{\mathcal{A}_{n'}\left(x\right)}\right)\right|
\]

Recall that $\mathcal{C}=\bigvee_{i=-\infty}^{0}{\displaystyle \widetilde{T}_{a}^{i}\mathcal{A}}$
denotes the $\sigma$-algebra in $\widetilde{\Omega}$ generated by
projection to the past. Since $\mathcal{C}\vee\mathcal{A}_{n}=\widetilde{T}_{a}^{n}\mathcal{C}$,
we have the equivariance relation $T_{a}^{n}\bigl(\bigl(\mu_{\omega}\bigr)_{\mathcal{A}_{n}\left(x\right)}\bigr)=\mu_{\widetilde{T}_{a}^{n}\left(\omega,x\right)}$
and also $\bigl(\mu_{\omega}\bigr)_{\mathcal{A}_{n}\left(x\right)}=R_{\theta_{\omega,x,n}}\bigl(S_{a^{-n}}\mu_{\widetilde{T}_{a}^{n}\left(\omega,x\right)}\bigr)$
for some phase $\theta_{\omega,x,n}$. This allows us to write,
\[
=\limsup_{N\rightarrow\infty}\left|\frac{1}{N}\sum_{n=1}^{N}\intop e_{m}\left(T_{b}^{n}R_{\theta_{\omega,x,n}}S_{a^{-n'}}\left(y\right)\right)d\mu_{\tilde{T}_{a}^{n'}\left(\omega,x\right)}\right|
\]

\begin{gather}
\leq{\displaystyle \limsup_{N\rightarrow\infty}}\frac{1}{N}\sum_{n=1}^{N}\left|\intop e_{m}\left(T_{b}^{n}R_{\theta_{\omega,x,n}}S_{a^{-n'}}\left(y\right)\right)d\mu_{\tilde{T}_{a}^{n'}\left(\omega,x\right)}\right|\label{eq:eval}
\end{gather}

If we split the integral above into the sum of integrals on the elements
of the uniform partition $\left\{ I_{j}^{k}\right\} _{j=1}^{k}$ such
that $\left|I_{j}^{k}\right|=1/k$, we get that for every interval
$I_{j_{0}}^{k}$ that doesn't contain a discontinuity we have,
\begin{gather*}
\intop_{I_{j_{0}}^{k}}e_{m}\left(T_{b}^{n}R_{\theta_{\omega,x,n}}S_{a^{-n'}}\left(y\right)\right)d\mu_{\tilde{T}_{a}^{n'}\left(\omega,x\right)}=\intop_{I_{j_{0}}^{k}}e_{m}\left(a^{z_{n}}y+\theta_{j_{0},\omega,x,n}\right)d\mu_{\tilde{T}_{a}^{n'}\left(\omega,x\right)}
\end{gather*}
where $\theta_{j_{0},\omega,x,n}$ is some phase that can be omitted
under absolute value. Otherwise $I_{j_{0}}^{k}$ contains a discontinuity
and its measure is less than $\sup_{j}\mu_{\omega}\left(I_{j}^{k}\right)$.
We take it into account in (\ref{eq:eval}), denoting $c_{\omega,k}=\left\lceil a/m_{b}\right\rceil \sup_{j}\mu_{\omega}\left(I_{j}^{k}\right)$
and write,
\[
\leq c_{\omega,k}+{\displaystyle \limsup_{N\rightarrow\infty}}\frac{1}{N}\sum_{n=1}^{N}\sum_{j=1}^{k}\left|\intop_{I_{j}^{k}}e_{m}\left(a^{z_{n}}y+\theta_{j,\omega,x,n}\right)d\mu_{\tilde{T}_{a}^{n'}\left(\omega,x\right)}\right|
\]

Now apply Theorem \ref{thm:Corollary-2.7}, after omitting the phases
because of the absolute value to obtain,
\begin{gather}
\leq c_{\omega,k}+\sum_{j=1}^{k}\intop\left|\intop_{I_{j}^{k}}e_{m}\left(a^{z}y\right)d\mu_{\eta}\right|d\nu_{\omega,x}\left(z,\eta\right)\label{eq:last-1-1}
\end{gather}

If we integrate both sides of the inequality (\ref{eq:first-1}) and
(\ref{eq:last-1-1}) w.r.t. $\widetilde{\mu}$, then by Corollary
\ref{thm:Corollary-2.7} it becomes,\\
\begin{dontbotheriftheequationisoverlong}
\intop\limsup_{N\rightarrow\infty}\left|\frac{1}{N}\sum_{n=1}^{m}e_{m}\left(T_{b}^{n}x\right)\right|d\widetilde{\mu}\left(\omega,x\right)\leq\intop\left(c_{\omega,k}+\sum_{j=1}^{k}\intop\left|\intop_{I_{j}^{k}}e_{m}\left(a^{z}y\right)d\mu_{\omega}\right|\right)dzd\widetilde{\mu}\left(\omega,x\right)
\end{dontbotheriftheequationisoverlong}%

Next, we apply the Cauchy-Schwartz inequality to get,
\begin{gather*}
\leq\intop\left(c_{\omega,k}+\sqrt{k\sum_{j=1}^{k}\intop\left|\intop_{I_{j}^{k}}e_{m}\left(a^{z}y\right)d\mu_{\eta}\right|^{2}dz}\right)d\widetilde{\mu}\left(\omega,x\right)
\end{gather*}
where Lemma \ref{lem:3.2} can provide the following evaluation: For
any $r>0$,

\[
\leq\intop\left(c_{\omega,k}+\sqrt{\frac{k\sum_{j=1}^{k}\mu_{\eta}\left(I_{j}^{k}\right)^{2}}{r\left|m\right|}+k\iint\chi_{B_{r}\left(y\right)}d\mu_{\omega}\left(y'\right)d\mu_{\omega}\left(y\right)}\right)d\widetilde{\mu}\left(\omega,x\right)
\]

Notice that the conditional measures $\mu_{\omega}$ are continuous
because $\mu$ was assumed to have positive entropy so by the dominated
convergence theorem,
\[
\intop c_{\omega,k}d\widetilde{\mu}\left(\omega,x\right)=\left\lceil \frac{a}{m_{b}}\right\rceil \intop\sup_{j}\mu_{\omega}\left(I_{j}^{k}\right)d\widetilde{\mu}\left(\omega,x\right)\xrightarrow{k\rightarrow\infty}0
\]
and similarly,
\[
\textbf{E}_{\mu_{\omega}}\left(\mu_{\omega}\left(B_{r}\left(y\right)\right)\right)=\iint\chi_{B_{r}\left(y\right)}\left(y'\right)d\mu_{\omega}\left(y'\right)d\mu_{\omega}\left(y\right)\xrightarrow{r\rightarrow0}0
\]

Finally, we can choose $r=\left|m\right|^{-1+\epsilon}$ and $k=\min\left\{ \left|m\right|^{0.5\epsilon},\textbf{E}_{\mu_{\omega}}\left(\mu_{\omega}\left(B_{r}\left(y\right)\right)\right)^{-1+\epsilon}\right\} $
for a small $\epsilon>0$ in a way that,

\[
\intop c_{\omega,k}+\sqrt{\frac{k\sum_{j=1}^{k}\mu_{\eta}\left(I_{j}^{k}\right)^{2}}{r\left|m\right|}+k\textbf{E}_{\mu_{\omega}}\left(\mu_{\omega}\left(B_{r}\left(y\right)\right)\right)}d\widetilde{\mu}\left(\omega,x\right)\xrightarrow{m\rightarrow\infty}0
\]

\hfill{}$\square$

\section{\label{part:Pf_quant_Raj}Proof of theorem \ref{thm:quant_Rajchman}}

The main difference between the previous proof and the quantitative
one here is that we assume explicit bounds on $\intop c_{\omega,k}d\widetilde{\mu}\left(\omega,x\right)$
and $\textbf{E}_{\mu_{\omega}}\left(\mu_{\omega}\left(B_{r}\left(y\right)\right)\right)$.
In addition, we aim for a pointwise decay instead of mean decay. This
may be achieved by relating the decay of a stochastic process to the
decay of its mean (as in Section \ref{subsec:lth power trick}).

We begin by repeating our assumptions. Let $a,b,\mu$ be as in Theorem
\ref{thm:Raj}. Denote the uniform partition of the interval $\rmz$
into $k$ pieces by $\left\{ I_{j}^{k}\right\} _{j=1}^{k}$ such that
$\left|I_{j}^{k}\right|=1/k$ and assume that for some $0<\alpha$
and for every $k\in\n$,
\begin{gather}
\underset{j,\omega}{\essup}\mu_{\omega}\left(I_{j}^{k}\right)\leq O\left(k^{-\alpha}\right)\label{eq:uni_con}
\end{gather}

When this holds it also imposes a secondary property which is important
for us,
\begin{align*}
\underset{\eta\in\Omega^{-}}{\essup}\iint\chi_{B_{k^{-1}}\left(y\right)}\left(x\right)d\mu_{\eta}\left(x\right)d\mu_{\eta}\left(y\right) & \leq O\left(k^{-\alpha}\right)
\end{align*}
but it is useful to have here a distinct parameter $\beta\geq\alpha$
such that,
\begin{equation}
\underset{\eta\in\Omega^{-}}{\essup}\iint\chi_{B_{k^{-1}}\left(y\right)}\left(x\right)d\mu_{\eta}\left(x\right)d\mu_{\eta}\left(y\right)\leq O\left(k^{-\beta}\right)\label{eq:sec_con}
\end{equation}

The first part of the proof is identical to the previous one but now
we think of the limit,
\[
\limsup_{N}\left|\frac{1}{N}\sum_{n=1}^{N}e_{m}\left(T_{b}^{n}x\right)\right|
\]
as a random variable from $\rmz$ to $\r$. Again, we fix a $\widetilde{\mu}$-typical
$\omega\in\Omega^{-}$ and a $\mu_{\omega}$-typical $x\in\left[0,1\right)$
and consider,
\[
\limsup_{N}\left|\frac{1}{N}\sum_{n=1}^{N}e_{m}\left(T_{b}^{n}x\right)\right|
\]
We repeat steps (\ref{eq:first-1}) to (\ref{eq:last-1-1}) from the
previous proof up to replace $c_{\omega,k}$ with $O\left(k^{-\alpha}\right)$
according to condition (\ref{eq:uni_con}). That is,\\
\begin{dontbotheriftheequationisoverlong}
\intop\limsup_{N}\left|\frac{1}{N}\sum_{n=1}^{N}e_{m}\left(T_{b}^{n}x\right)\right|d\widetilde{\mu}\left(\omega,x\right)\leq\intop O\left(k^{-\alpha}\right)+\sqrt{k\sum_{j=1}^{k}\intop\left|\intop_{I_{j}^{k}}e_{m}\left(a^{z}y\right)d\mu_{\eta}\right|^{2}dz}\ d\widetilde{\mu}\left(\omega,x\right)
\end{dontbotheriftheequationisoverlong}%

In this case Lemma \ref{lem:3.2} can provide the following evaluation:
For any $r>0$, using $\left|I_{j}^{k}\right|=1/k$ and $\mu_{\eta}\left(I_{j}^{k}\right)\leq O\left(k^{-\alpha}\right)$,

\[
\leq\intop O\left(k^{-\alpha}\right)+\sqrt{\frac{O\left(k^{1-\alpha}\right)}{r\left|m\right|}+k\int\chi_{B_{r}}\left(y\right)d\mu_{\omega}\left(y'\right)d\mu_{\omega}\left(y\right)}d\widetilde{\mu}\left(\omega,x\right)
\]

Finally, denote $r=k^{-\gamma}$ and choose $k=\bigl[\left|m\right|^{\delta}\bigr]$
for some $0<\delta$. Then with condition (\ref{eq:sec_con}),
\begin{gather*}
\leq\intop O\left(\left|m\right|^{-\delta\alpha}\right)+\sqrt{O\left(\left|m\right|^{-1-\delta\left(\alpha-1\right)+\gamma\delta}\right)+O\left(\left|m\right|^{\delta\left(1-\gamma\beta\right)}\right)}d\widetilde{\mu}\left(\omega,x\right)\\
\leq O\left(\left|m\right|^{\min_{\gamma,\delta>0}\max\left\{ -\delta\alpha,\frac{-1-\delta\left(\alpha-1\right)+\gamma\delta}{2},\frac{\delta\left(1-\gamma\beta\right)}{2}\right\} }\right)
\end{gather*}

Notice that $\frac{-1-\delta\left(\alpha-1\right)+\gamma\delta}{2}\geq-0.5$
which means that the referred expression in the exponent is bounded
from below by $-0.5$. An upper bound is achieved under the equality
$2\delta\alpha=1+\delta\left(\alpha-1\right)-\gamma\delta=\gamma\delta\beta-\delta$
that leads to $\gamma=\frac{2\alpha+1}{\beta}$ and $\delta=\frac{\beta}{\beta\left(1+\alpha\right)+2\alpha+1}$
. Therefore,
\[
\leq O\left(\left|m\right|^{-\left(\frac{\alpha\beta}{\beta\left(1+\alpha\right)+2\alpha+1}\right)}\right)
\]
and the conclusion is that the exponent is strictly bounded between
$-0.5$ and $0$. As already mentioned, this decay rate is too slow
for applying our second strategy from the introduction.%
\hfill{}$\square$

We remark that both conditions \ref{eq:uni_con} and \ref{eq:sec_con}
can be relaxed in several ways and mention here one of them as an
example. If conditions \ref{eq:uni_con} and \ref{eq:sec_con} hold
outside a sequence of measurable sets whose measure decays in some
known rate we still promise the mean decay of the Fourier transform
in a rate that we are able to bound.

Lastly, although we didn't achieve the desired result, we want to
make a suggestion of how to progress further in that direction. This
will be relevant if someone managed to improve the evaluation of the
decay rate from above,

\subsection{\label{subsec:lth power trick} Relating the decay rate of a stochastic
process to the decay rate of its mean }

Let $\left(X,\mathcal{B},\mu\right)$ be a probability space and $\bigl\{ A_{m}\bigr\}_{m=1}^{\infty}$
be a stochastic process. A conventional application of Markov's inequality
and the first Borel-Cantelli lemma implies that the decay rate in
$m$ of $\bigl\{\bigl|A_{m}\bigr|\bigr\}_{m=1}^{\infty}$ relates
to the decay rate in  $m$ of its means $\bigl\{\textbf{E}_{\mu}\bigl(\bigl|A_{m}\bigr|\bigr)\bigr\}_{m=1}^{\infty}$
in the following way,
\begin{prop}
If there exists a real $0<\alpha$ such that $\textbf{E}_{\mu}\bigl|A_{m}\bigr|<O\left(m^{-1-\alpha}\right)$
then for every $0<\epsilon<\alpha$ a.s. $\bigl|A_{m}\bigr|\leq o\left(m^{-\alpha+\epsilon}\right)$.
\end{prop}

\begin{proof}
Fix $c>1$. By Markov's inequality $\pr_{\mu}\left\{ \left|A_{m}\right|>\frac{1}{c}m^{-\alpha+\epsilon}\right\} <\frac{c\textbf{E}_{\mu}\bigl|A_{m}\bigr|}{m^{-\alpha+\epsilon}}$
and from the hypothesis above there exists some $s>0$ s.t. $\pr_{\mu}\left\{ \left|A_{m}\right|>\frac{1}{c}m^{-\alpha+\epsilon}\right\} <\frac{c\cdot s\cdot m^{-1-\alpha}}{m^{-\alpha+\epsilon}}$.
Thus,
\begin{gather*}
\pr_{\mu}\left\{ \left|A_{m}\right|>\frac{1}{c}m^{-\alpha+\epsilon}\right\} <\frac{c\cdot s\cdot m^{-1-\alpha}}{m^{-\alpha+\epsilon}}=c\cdot s\cdot m^{-1-\epsilon}\\ \implies\sum_{m}\pr_{\mu}\left\{ \left|A_{m}\right|>\frac{1}{c}m^{-\alpha+\epsilon}\right\} <\infty
\end{gather*}
Finally, by the first Borel-Cantelli lemma $\pr_{\mu}\bigl(\limsup_{m}\bigl\{\bigl|A_{m}\bigr|>\frac{1}{c}m^{-\alpha+\epsilon}\bigr\}\bigr)=0$
which implies that $\bigl\{\bigl|A_{m}\bigr|>\frac{1}{c}m^{-\alpha+\epsilon}\bigr\}$
occur finitely often with probability $1$.
\end{proof}

\section{\label{part:The-invariance-of}Proof of theorem \ref{thm:xb-inv.}}

We begin with providing an equivalent definition of weak-{*} convergence
that will be needed in that section, via the next continuous mapping
theorem (see Theorem 2.57 in \cite{key-18}),
\begin{thm}
\label{thm:Continuous mapping theorem} $\mu_{n}\xrightarrow{w-*}\mu$
iff $\intop fd\mu_{n}\xrightarrow{n\rightarrow\infty}\intop fd\mu$
for every bounded function $f:X\rightarrow\mathbb{\r}$ with $\mu\left(\left\{ x:f\text{ has a discontinuity at }x\right\} \right)=0$.
\end{thm}

Let $x$ be a $\mu$-typical point. In this part we assume that $x$
equidistributes along a subsequence $\left\{ N_{k}\right\} $ under
$T_{b}$ for some measure where all the hypotheses are as in Problem
\ref{prob:(Generalized-HET)} and also that $\left\{ 0\right\} $
is not an atom of this measure. Recall that in our notations it means
that the limit $\lambda_{x,\left\{ N_{k}\right\} }=\lim_{k\rightarrow\infty}\frac{1}{N_{k}}\sum_{i=0}^{N_{k}-1}\delta_{T_{b}^{i}\left(x\right)}$
is well defined. Then we show that $\lambda_{x,\left\{ N_{k}\right\} }$
must be $T_{b}$-invariant. For convenience, our proof is given as
if the equidistribution is along the whole sequence since for equidistribution
along a subsequence the argument is identical. Here we denote the
limit measure by $\lambda_{\infty}$ for short.

Fix $1<b\in\r$ and let $x\in\rmz$ be a $\mu$-typical point and
denote,
\begin{gather*}
\lambda_{N}=\frac{1}{N}\sum_{i=0}^{N-1}\delta_{T_{b}^{i}x}\xrightarrow{w-*}\lim_{N\rightarrow\infty}\frac{1}{N}\sum_{i=0}^{N-1}\delta_{T_{b}^{i}x}=\lambda_{\infty}
\end{gather*}

We need to show that 
\[
\intop fd\lambda_{\infty}=\intop fdT_{b*}\lambda_{\infty}
\]
for all $f\in C\left(\rmz\right)$.

We have,
\[
\lambda_{N}-T_{b*}\lambda_{N}=\frac{1}{N}\sum_{i=0}^{N-1}\delta_{T_{b}^{i}\left(x\right)}-\frac{1}{N}\sum_{i=0}^{N-1}\delta_{T_{b}^{i+1}\left(x\right)}=\frac{\delta_{x}-\delta_{T_{b}^{N}\left(x\right)}}{N}\xrightarrow{N\rightarrow\infty}0
\]
 so $\lim_{N\rightarrow\infty}\left(\lambda_{N}-T_{b*}\lambda_{N}\right)=0$.

Additionally, for every $f\in C\left(\rmz\right)$, $f\circ T_{b}$
has at most one discontinuity which is located at $0$ and by assumption
$\lambda_{\infty}\left(\left\{ 0\right\} \right)=0$. Since $\lambda_{N}\xrightarrow{w-*}\lambda_{\infty}$
for any $f\in C\left(\rmz\right)$ then by Theorem \ref{thm:Continuous mapping theorem},
\begin{gather*}
\textbf{E}_{T_{b*}\lambda_{N}}\left(f\right)=\textbf{E}_{\lambda_{N}}\left(f\circ T_{b}\right)\xrightarrow{N\rightarrow\infty}\textbf{E}_{\lambda_{\infty}}\left(f\circ T_{b}\right)=\textbf{E}_{T_{b*}\lambda_{\infty}}\left(f\right)
\end{gather*}
Now we can stitch it all together and get for every $f\in C\left(\rmz\right)$
that,\\
\begin{align*}
\left|\textbf{E}_{\lambda_{\infty}}\left(f\right)-\textbf{E}_{T_{b*}\lambda_{\infty}}\left(f\right)\right|\leq & \left|\textbf{E}_{\lambda_{\infty}}\left(f\right)-\textbf{E}_{\lambda_{N}}\left(f\right)\right|+\\
+ & \left|\textbf{E}_{\lambda_{N}}\left(f\right)-\textbf{E}_{T_{b*}\lambda_{N}}\left(f\right)\right|+\\
+ & \left|\textbf{E}_{T_{b*}\lambda_{N}}\left(f\right)-\textbf{E}_{T_{b*}\lambda_{\infty}}\left(f\right)\right|\rightarrow0
\end{align*}
as $N$ tends to infinity which implies the desired result\hfill{}$\square$

\section{\label{part:range of validity}Exploring conditions (\ref{eq:uni_con})
and (\ref{eq:sec_con})}

In this section we inquire the validity range of Conditions (\ref{eq:uni_con})
and (\ref{eq:sec_con}). We added some brief theoretical background
that will be used for that subjective.

\subsection{\label{subsec:Stationary-coding}Stationary coding}

The content of this subsection is taken from the book of P. Shields
on ergodic theory \cite{key-14}

Let $A$ and $B$ be finite sets. A Borel measurable map $F:A^{\z}\rightarrow B^{\z}$
is a \emph{stationary coder} if $F\left(T_{A}x\right)=T_{B}F\left(x\right)$
for every $x\in A^{\z}$ where $T_{A}$ and $T_{B}$ denotes the shifts
on the respective two-sided sequence spaces $A^{\z}$ and $B^{\z}$.
Stationary coding carries a Borel measure $\mu$ on $A^{\z}$ into
the measure $\nu=\mu\circ F^{-1}$ defined for Borel subsets $C$
of $B^{\z}$ by $\nu\left(C\right)=\mu\left(F^{-1}\left(C\right)\right)$.
The encoded process $\nu$ is said to be a \emph{stationary coding}
of $\mu$, with coder $F$. It is immediate to conclude that a stationary
coding of a stationary process is itself a stationary process. 

The map $f:A^{\z}\rightarrow B$ defined by the formula $f\left(x\right)=F\left(x\right)_{0}$
is the \emph{time-zero coder} associated with $F$. Notice that stationary
coding preserves ergodicity.

Let $T$ be an invertible, ergodic transformation of the probability
space $\left(X,\Sigma,\mu\right)$. Let $S$ be a set of positive
measure. Denote by $\left\{ R_{n}\right\} $ the $\left(T,\mathcal{A}_{S}\right)$-process
defined by the partition $\mathcal{A}_{S}=\left\{ S,X\setminus S\right\} $
with $X\setminus S$ labeled by $0$ and $S$ labeled by $1$. That
is, $R_{n}\left(x\right)=\chi_{S}\left(T^{n}x\right)$ for every integer
$n\in\z$. Notice that $\left\{ R_{n}\right\} $ is a stationary coding
of the $\left(T,\mathcal{A}_{S}\right)$-process with time zero coder
$\chi_{S}$ and therefore it is ergodic.

Notice that such a non trivial $\left\{ R_{n}\right\} $ process (i.e
$0<\mu\left(S\right)<1$) has positive entropy where $\left(X,\Sigma,\mu\right)$
is a product measure on $\mathcal{A}_{S}^{\mathbb{Z}}$. This is because
that a non trivial $\left(T,\left\{ S,X\setminus S\right\} \right)$-process
is a factor of Bernoulli measure which has completely positive entropy.

\subsection{Reverse Markov's inequality}

The next reverse version of Markov's inequality is famous and easy
to prove, we present it here with a proof for completeness.
\begin{thm}
\label{thm:(Reverse-Markov-inequality)}(Reverse Markov's inequality)
Let $X$ be a random variable on a probability space $\left(\Omega,\pr\right)$
that satisfies $\pr\left(X\leq a\right)=1$ for some constant $a$.
Then, for $d<\ev X$
\[
\pr\left(X>d\right)\geq\frac{\ev X-d}{a-d}
\]
\end{thm}

\begin{proof}
Define the random variable $V=a-X$ which is a.s. nonnegative by assumption.
The event $\left\{ X\leq d\right\} $ is equivalent to the event $\left\{ V\geq a-d\right\} $.
Now with Markov's inequality,
\[
\pr\left(\left\{ X\leq d\right\} \right)=\pr\left(\left\{ V\geq a-d\right\} \right)\leq\frac{\ev V}{a-d}=\frac{a-\ev X}{a-d}
\]
where the RHS numerator and denominator are strictly positive since
$d<\ev X\leq a$. Finally,
\[
\pr\left(\left\{ X>d\right\} \right)=1-\pr\left(\left\{ X\leq d\right\} \right)>\frac{\ev X-d}{a-d}
\]
\end{proof}

\subsection{\label{subsec:A-counterexample}Not all the $\mu$'s meet condition
(\ref{eq:sec_con})}

Here we construct a counterexample which violates the weaker condition
(\ref{eq:sec_con}). We anticipate that condition (\ref{eq:uni_con})
can be violated by simpler examples. 

We use the formalism from Sections \ref{sec:Specifics-in-ergodic and entropy}
and \ref{subsec:Stationary-coding}. Let $\left(X_{l},\Sigma,\nu,\sigma\right)$
be a uniformly distributed full Bernoulli shift in $2<l$ symbols.
We define a measurable set $\mathcal{Y}\in\Sigma$ recursively.

Fix $0<\epsilon<0.5$. For $k=1$ fix $2\leq n_{1}\in\n$ such that
$\log\left(n_{1}\right)^{-1}<1-\epsilon$ and let $Y_{1}$ be a set
such that $0.5\log\bigl(n_{1}\bigr)^{-2}<\nu\bigl(Y_{1}\bigr)<\log\bigl(n_{1}\bigr)^{-2}$.
Define $\mathcal{Y}_{1}=\bigcup_{m=0}^{\left[\log n_{1}\right]}\sigma^{-m}Y_{1}$
so that $0<\nu\bigl(\mathcal{Y}_{1}\bigr)\leq\log\bigl(n_{1}\bigr)^{-1}$
and also $\nu\bigl(\mathcal{Y}_{1}\bigr)<1-\epsilon$.

For $1<k\in\n$ choose $n_{k-1}<n_{k}\in\n$ such that $\log\bigl(n_{k}\bigr)^{-1}<1-\epsilon-\sum_{i=1}^{k-1}\nu\bigl(\mathcal{Y}_{i}\bigr)$
 and let $Y_{k}$ be a set such that $0.5\cdot\log\bigl(n_{k}\bigr)^{-2}<\nu\bigl(Y_{k}\bigr)<\log\bigl(n_{k}\bigr)^{-2}$.
Define $\mathcal{Y}_{k}=\bigcup_{m=0}^{\left[\log n_{k}\right]}\sigma^{-m}Y_{k}$
so that $0<\nu\bigl(\mathcal{Y}_{k}\bigr)\leq\frac{1}{\log n_{k}}\leq1-\epsilon-\sum_{i=1}^{k-1}\nu\bigl(\mathcal{Y}_{i}\bigr)$
and therefore $0<\sum_{i=1}^{k}\nu\bigl(\mathcal{Y}_{i}\bigr)<1-\epsilon$.

Finally, define $\mathcal{Y}=\bigcup_{i=1}^{\infty}\mathcal{Y}_{i}$
and define a process on $\left\{ 0,1\right\} ^{\mathbb{Z}}$ by $R_{n}\left(x\right)=\chi_{\mathcal{Y}}\left(\sigma^{n-1}x\right)$
with the induced measure $\widetilde{\mu}=\nu\circ R^{-1}$ as in
Section \ref{subsec:Stationary-coding}. By our construction $0<\nu\left(\mathcal{Y}\right)<1-\epsilon$
and $\left\{ n_{k}\right\} $ is strictly increasing. Since $\left(X_{l},\Sigma,\nu\right)$
is ergodic so is $\left\{ R_{n}\right\} $ as a stationary coding
with time zero coder (which trivially also preserves the measure).
That is, we can interpret $\widetilde{\mu}$ as a probability measure
which is ergodic and shift invariant on the binary full shift. It
also has positive entropy by \ref{subsec:Stationary-coding}.

In another direction, notice that $\left[0,1\right)$ is the image
of the infinite binary sequences $\left\{ 0,1\right\} ^{\n}$ under
the natural map $\left(x_{i}\right)_{i=1}^{\infty}\mapsto\sum_{i=1}^{\infty}\frac{x_{i}}{2^{i}}$
which is also a bijection on the complement of a countable set. That
is, a bijection on the complement of a null set for every continuous
measure. Thus, if we denote $X_{2}=\left\{ 0,1\right\} ^{\z}$ (the
2-shift) then we can define up to a null set the bijection $\iota:X_{2}\rightarrow\left\{ 0,1\right\} ^{-\mathbb{N}_{0}}\times\left[0,1\right)$.
We can also define the pushforward $\widetilde{\rho}=\widetilde{\mu}\circ\iota^{-1}$
(that is, a measure on $\left\{ 0,1\right\} ^{-\mathbb{N}_{0}}\times\left[0,1\right)$)
which preserves all the relevant properties of $\widetilde{\mu}$
(ergodicity is trivial and positive entropy by being a non trivial
factor of system with completely positive entropy). Now, $\widetilde{\rho}=\intop\rho_{\eta}d\widetilde{\rho}\left(\eta\right)$
where it is disintegrated with respect to the $\sigma$-algebra that
generated by projection to the past (like in Section \ref{subsec:Results})
and the conditional measures $\rho_{\eta}$ identified as measures
on $\left[0,1\right)$. Similarly, $\widetilde{\mu}=\intop d\mu{}_{\eta}d\widetilde{\mu}\left(\eta\right)$
where $\eta\in\left\{ 0,1\right\} ^{-\n_{0}}$.

For every $k\in\n$ denote the event $\bigcap_{i=0}^{\left[\log n_{k}\right]}\left\{ R_{i}=1\right\} $
by $E_{k}$. We begin with,
\[
\intop\iint\chi_{B_{n_{k}^{-1}}\left(y\right)}\left(x\right)d\rho_{\eta}\left(x\right)d\rho_{\eta}\left(y\right)d\widetilde{\rho}\left(\eta\right)
\]
 where $B_{n_{k}^{-1}}\left(y\right)=\left\{ x\in\left[0,1\right):\left|x-y\right|<n_{k}^{-1}\right\} $.
We need to show that this integral decays in $n_{k}$ with sub polynomial
rate. We can pullback this integration to the space $X_{2}$ and restrict
it to the event $E_{k}\times E_{k}$ since every pair $\left(x',y'\right)\in E_{k}\times E_{k}$
corresponds to a pair $\left(x,y\right)\in\left[0,1\right)^{2}$ with
$\left|x-y\right|<\frac{1}{n_{k}}$. Thus,
\begin{align*}
 & \geq\intop\iintop\chi_{E_{k}\times E_{k}}\left(x,y\right)d\mu_{\eta}\times\mu_{\eta}\left(x,y\right)d\widetilde{\mu}\left(\eta\right)\\
 & =\intop\intop\chi_{E_{k}}\left(x\right)d\mu{}_{\eta}\left(x\right)\intop\chi_{E_{k}}\left(y\right)d\mu_{\eta}\left(y\right)d\widetilde{\mu}\left(\eta\right)\\
 & =\intop\left(\intop\chi_{E_{k}}\left(x\right)d\mu{}_{\eta}\left(x\right)\right)^{2}d\widetilde{\mu}\left(\eta\right)\\
 & \geq\left(\intop\intop\chi_{E_{k}}\left(x\right)d\mu{}_{\eta}\left(x\right)d\widetilde{\mu}\left(\eta\right)\right)^{2}\\
 & =\left(\intop\chi_{E_{k}}\left(x\right)d\widetilde{\mu}\left(x\right)\right)^{2}
\end{align*}
where in the first equality we used Fubini and then split the integral
into two separate integrals, and the second inequality is Jensen.
We can pullback the integration once again to the full shift $X_{l}$
and restrict it to the event $Y_{k}$ which included in the preimage
of $E_{k}$,

\begin{align*}
 & \geq\left(\intop\chi_{Y_{k}}\left(x\right)d\nu\left(x\right)\right)^{2}\\
 & =0.25\cdot\log\bigl(n_{k}\bigr)^{-4}
\end{align*}
Now, for every $k\in\n$ we can use the reverse Markov inequality
(Theorem \ref{thm:(Reverse-Markov-inequality)}) with $a=2$ and $d=\frac{1}{8\log\left(n_{k}\right)^{4}}$
to get,
\[
\pr_{\widetilde{\mu}}\left(\bigl\{\intop\mu{}_{\eta}\left(B_{n_{k}^{-1}}\left(y\right)\right)d\mu{}_{\eta}\left(y\right)\geq\frac{1}{8\log\left(n_{k}\right)^{4}}\bigr\}\right)\geq\frac{1}{16\log\left(n_{k}\right)^{4}}
\]

This violates condition (\ref{eq:sec_con}).

\subsection{\label{subsec:Finite-memory-length}Processes with memory of finite
length meet condition (\ref{eq:uni_con})}

In this special case it will be enough to use the total probability
formula to reach condition (\ref{eq:uni_con}) which also implies
that condition (\ref{eq:sec_con}) holds.

Let $A$ be a finite alphabet $\left|A\right|=l$ for some integer
$l\geq2$ and let $\left(A^{\z},\mu,T\right)$ be a finite memory
length symbolic process which is ergodic, invariant and with positive
entropy w.r.t. $T$. That is, the prediction of $A_{0}$ depends only
on a finite portion of the past $A_{-n},\dots,A_{-1}$ for some $n<\infty$
or explicitly for every $m\in\z$ and $x\in A^{\z}$, $\mu\left(\left[x_{m+n}^{\infty}\right]|\left[x_{-\infty}^{m}\right]\right)=\mu\left(\left[x_{m+n}^{\infty}\right]\right)$.
Denote $s=\max_{a\in A}\mu\left(\left[a\right]\right)<1$. Assume
that $n<m$ then with direct computation,
\begin{align*}
\mu_{\left[x_{-\infty}^{0}\right]}\left(\left[x_{0}^{m}\right]\right) & \leq\mu\left(\left[x_{n}^{m}\right]\right)\\
 & =\mu\left(\left[x_{n}\right]\right)\mu\left(\left[x_{n+1}^{m}\right]|\left[x_{n}\right]\right)\\
 & \leq\mu\left(\left[x_{n}\right]\right)\mu\left(\left[x_{2n}^{m}\right]\right)\\
 & \vdots\\
 & =\prod_{i=1}^{\fr{m/n}}\mu\left(\left[x_{n\cdot i}\right]\right)\\
 & \leq s^{m/n+1}
\end{align*}

If we consider a uniform partition of the unit interval $\left\{ I_{j}^{k}\right\} _{j=1}^{k}$
such that $\left|I_{j}^{k}\right|=\frac{1}{k}$ and where $k=l^{m}$
then the display above implies that $\essup_{j,\omega}\mu_{\omega}\left(I_{j}^{k}\right)\leq O\left(k^{\log\left(s\right)/\left(n\log\left(l\right)\right)}\right)$.

\section{\label{part:Rajchman+invariant is insufficient}Rajchman and invariance
does not force uniqueness}

\subsection{Self-similarity and classes of algebraic numbers}

Here we set out some of the basic definitions regarding self similarity.
A set $\mathcal{F}=\left\{ f_{1},\dots f_{k}\right\} $ for $k\geq2$
of contractions on $\r$, $f_{i}\left(x\right)=r_{i}x+a_{i}$, with
$0<r_{i}<1$ for each $i\in\left\{ 1,\dots,k\right\} $, is an\textbf{
}\emph{iterated function system} or \emph{IFS} for short. We also
call the $f_{i}$'s \emph{similarities}.

A key fact in this topic is that there exists a unique non-empty compact
set $K\subset\r$ such that $K=\bigcup_{i=1}^{k}f_{i}\left(K\right)$.
We call it the \emph{attractor}\textbf{ }or the\textbf{ }\emph{self-similar
set} of the IFS $\mathcal{F}$.

Given a list of positive numbers $\boldsymbol{p}=\left(p_{1},\dots,p_{k}\right)$
with $\sum_{i=1}^{k}p_{i}=1$ we call it a \emph{positive probability
vector} and there is a unique probability measure $\mu_{\boldsymbol{p}}$
with $\mu_{\boldsymbol{p}}=\sum_{i=1}^{k}p_{i}\cdot f_{i}\mu_{\boldsymbol{p}}$.
This measure is supported on the attractor of $\mathcal{F}$ and called
a \emph{self-similar measure}.

A \emph{Pisot number} is a real algebraic integer greater than $1$
all of whose Galois conjugates are less than $1$ in absolute value.
A \emph{Salem number} is a real algebraic integer greater than $1$
whose all conjugate roots have absolute value no greater than $1$,
and at least one of them has an absolute value which equals $1$. There
are countably many Pisot and Salem numbers.

We conclude this part with a recent theorem of Varj\'u and Yu \cite{key-16}
which plays a major role in our proof that $T_{b}$-invariant Rajchman
measure need not be Parry,
\begin{thm}
\label{thm:1.6} Let $k\geq2$ be an integer. Let $r_{1}=r^{l_{1}},\dots,r_{k}=r^{l_{k}}$
for some $r\in\left(0,1\right)$ and $l_{1},\dots,l_{k}\in\z_{>0}$
with gcd$\left(l_{1},\dots,l_{k}\right)=1$. Assume that $r^{-1}$
is not a Pisot or Salem number. Let $\mu$ be a non-singleton self-similar
measure associated to the IFS $\mathcal{F}=\left\{ f_{1},\dots,f_{k}\right\} $
($f_{i}\left(x\right)=r_{i}x+a_{i}$) and a positive probability measure.
Then,
\[
\left|\hat{\mu}\left(z\right)\right|=O\left(\left|\log\left(z\right)^{-c}\right|\right)
\]
for some $c>0.$
\end{thm}

We use Theorem \ref{thm:1.6} to prove that Rajchman $T_{b}$-invariant
measure must not be unique. Let $b\geq2$ be a specified non simple
number which is not Salem or Pisot. We can guarantee the existence
of such a number by cardinality considerations (recall that Pisot numbers
as well as Salem numbers and simple numbers are countable while specified
numbers have the cardinality of continuum).

Denote $\n_{0}=\n\cup\left\{ 0\right\} $ for short and denote the
$b$-expansion of $b$ by $\left(a_{i}\right)$. That is $b=a_{0}+\frac{a_{1}}{b}+\dots$
and in particular $a_{0}=\fr b$. Every sequence $\left(x_{i}\right)\in\left\{ 0,1\right\} ^{\n_{0}}$
and $n\in\n$ satisfies $\left(x_{i+n}\right)<\left(a_{i}\right)$
since $x_{n}\in\left\{ 0,1\right\} $ which is either way less than
$\fr b=a_{0}$. Recall that $X_{b}\subset\Lambda_{b}^{\n_{0}}$ is
the subset of sequences that encodes $b$-expansions, so by Parry's
criterion (Theorem \ref{thm:Parry's criterion}) we get that $\left\{ 0,1\right\} ^{\n_{0}}\subset X_{b}$. 

Define the set $K=\left\{ \sum_{i=0}^{\infty}\frac{b_{i}}{b^{i}}:\left(b_{i}\right)\in\left\{ 0,1\right\} ^{\n_{0}}\right\} $
which is the image of $\left\{ 0,1\right\} ^{\n}$ under the $b$-expansion
and define the IFS $\mathcal{F}=\left\{ f_{0}\left(x\right)=\frac{x}{b},f_{1}\left(x\right)=\frac{x}{b}+\frac{1}{b}\right\} $.
If we denote concatenation of symbols by $\centerdot$ then clearly
$\left\{ 0,1\right\} ^{\n_{0}}=0\centerdot\left\{ 0,1\right\} ^{\n_{0}}\cup1\centerdot\left\{ 0,1\right\} ^{\n_{0}}$
which means under the $b$-expansion that $K=\cup_{i=0}^{1}f_{i}\left(K\right)$.
Thus, $K$ is the attractor of $\mathcal{F}$ and for every positive
probability vector $\boldsymbol{p}=\left(p_{0},p_{1}\right)$ the
self similar measure that satisfies $\mu_{\boldsymbol{p}}=\sum_{i=0}^{1}p_{i}\cdot f_{i}\mu_{\boldsymbol{p}}$
is well defined. Pulling it back again under the $b$-expansion, it
translated into Bernoulli shift $\boldsymbol{p}^{\n_{0}}$ on $X_{b}|_{\left\{ 0,1\right\} ^{\n_{0}}}=\left\{ 0,1\right\} ^{\n_{0}}$
where the former is $T_{b}$-invariant and correspondingly the later
is shift invariant.

If we consider Theorem \ref{thm:1.6} where both of the $l_{i}$'s
are equal to $1$ such that $gcd\left(l_{0},l_{1}\right)=1$ then
all the hypotheses hold w.r.t $b$ and $\mathcal{F}$ and any self-similar
measure $\mu_{\boldsymbol{p}}$ provides the Rajchman property,
\[
\left|\hat{\mu}_{\boldsymbol{p}}\left(z\right)\right|=O\left(\left|\log\left(z\right)^{-c}\right|\right)
\]
Since there are infinitely many of them the uniqueness is violated.

\end{document}